\documentclass[12pt]{article}
\usepackage{amsmath,amssymb,amsthm,color} 
\usepackage[unicode,breaklinks=true,colorlinks=true]{hyperref}

\numberwithin{equation}{section}
\newtheorem{theorem}{Theorem}[section]

\newtheorem{lemma}[theorem]{Lemma}
\newtheorem{definition}[theorem]{Definition} 

\newtheorem{assumption}[theorem]{Assumption}

\theoremstyle{remark}
\newtheorem{remark}[theorem]{Remark}

\newcommand{\bke}[1]{\left( #1 \right)}
\newcommand{\bkt}[1]{\left[ #1 \right]}
\newcommand{\bket}[1]{\left\{ #1 \right\}}
\newcommand{\norm}[1]{\| #1 \|}

\newcommand{\bka}[1]{\left\langle #1 \right\rangle}

\newcommand{\al}{\alpha}

\newcommand{\e}{\epsilon}

\newcommand{\la}{\lambda}

\newcommand{\om}{{\omega}}
\newcommand{\si}{\sigma}

\renewcommand{\th}{\theta}

\newcommand{\R}{{\mathbb R }}\newcommand{\RR}{{\mathbb R }}
\newcommand{\N}{{\mathbb N}}
\newcommand{\Z}{{\mathbb Z}}
\newcommand{\cR}{{\mathcal R}}
\newcommand{\pd}{{\partial}}
\newcommand{\nb}{{\nabla}}
\newcommand{\lec}{\lesssim}

\newcommand{\I}{\infty}
 
\renewcommand{\div}{\mathop{\mathrm{div}}}

\def\si{\sigma}
\newcommand{\donothing}[1]{{}}

\newcommand{\EQ}[1]{\begin{equation}\begin{split} #1 \end{split}\end{equation}}

\DeclareMathOperator*{\esssup}{ess\,sup}

\makeatletter
\newcommand{\xRightarrow}[2][]{\ext@arrow 0359\Rightarrowfill@{#1}{#2}}
\makeatother

\begin{document}
\title{Forward discretely self-similar solutions of the Navier-Stokes equations II} 
\author{Zachary Bradshaw and Tai-Peng Tsai}
\date{\today}
\maketitle
\begin{abstract}
For any discretely self-similar, incompressible initial data $v_0$ which satisfies $\|v_0  \|_{L^3_w(\R^3)}\leq c_0$ where $c_0$ is allowed to be large, we construct a forward discretely self-similar local Leray solution in the sense of Lemari\'e-Rieusset to the 3D Navier-Stokes equations in the whole space.  No further assumptions are imposed on the initial data; in particular, the data is not required to be continuous or locally bounded on $\R^3\setminus \{0\}$.  The same method gives a third construction of self-similar solutions (after those in \cite{JiaSverak} and \cite{KT-SSHS}) and works for any $-1$-homogeneous initial data in $L^3_w$. 
\end{abstract}

\section{Introduction}
\label{sec1}

In three dimensions, the Navier-Stokes equations are
\begin{equation} 
\begin{array}{ll}\label{eq:NSE}
 \partial_t v -\Delta v +v\cdot\nabla v+\nabla \pi  = 0&\mbox{~in~}\R^3\times [0,\infty)
\\  \nabla\cdot v = 0&\mbox{~in~}\R^3\times [0,\infty),
\end{array}
\end{equation}
and are supplemented with some initial data $v_0$.  The existence of weak solutions for finite energy initial data ($v_0\in L^2$) was developed by Leray in \cite{leray} and later generalized by Hopf in \cite{hopf}, based on the a priori bound
\begin{equation}
\int |v(x,t)|^2dx + \int_0^t \int 2 |\nabla v(x,t')|^2 dx\,dt'\le \int |v_0(x)|^2dx ,
\end{equation}
the same bound as the Stokes system which is \eqref{eq:NSE} without the nonlinearity $v\cdot\nabla v$. Another important property of \eqref{eq:NSE} is its natural scaling: given a solution $v$ and $\lambda>0$, it follows that 
\begin{equation}
	v^{\lambda}(x,t)=\lambda v(\lambda x,\lambda^2t),
\end{equation}
is also a solution with associated pressure 
\begin{equation}
	\pi^{\lambda}(x,t)=\lambda^2 \pi(\lambda x,\lambda^2t),
\end{equation}
and initial data 
\begin{equation}
v_0^{\lambda}(x)=\lambda v_0(\lambda x).
\end{equation}
A solution is self-similar (SS) if it is scaling invariant with respect to this scaling, i.e.~if $v^\lambda(x,t)=v(x,t)$ for all $\lambda>0$. If this scale invariance holds for a particular $\lambda>1$, then we say $v$ is discretely self-similar with factor $\lambda$ (i.e.~$v$ is $\lambda$-DSS). Similarly $v_0$ can be SS or $\lambda$-DSS.  The class of DSS solutions contains the SS solutions since any SS $v$ is $\lambda$-DSS for any $\lambda>1$. Recall that $L^3_w(\R^3)$ is the weak Lebesgue space which is equivalent to the Lorenz space $L^{(q,r)}(\R^3)$ with $(q,r)=(3,\infty)$. The natural spaces to study SS $v_0$ and $v$ are $L^3_w(\R^3)$ and $L^\infty(0,\infty;L^3_w(\R^3))$.  

Self-similar solutions are determined by the behavior at any fixed time. This leads to an ansatz of $v$ in terms of a time-independent profile $u$, namely,
\begin{equation}\label{ansatz1}
v(x,t) = \frac 1 {\sqrt {2t}}\,u\bigg(\frac x {\sqrt{2t}}\bigg), 
\end{equation} 
where $u$ solves the \emph{Leray equations}
\begin{equation} 
\begin{array}{ll}\label{eq:stationaryLeray}
 -\Delta u-u-y\cdot\nabla u +u\cdot \nabla u +\nabla p = 0&\mbox{~in~}\R^3
\\  \nabla\cdot u=0&\mbox{~in~}\R^3,
\end{array}
\end{equation}
in the variable $y=x/\sqrt t$.
Similarly, $\lambda$-DSS solutions are decided by their behavior on the time interval $1\leq t\leq \lambda^2$ and we have
\begin{equation}\label{ansatz2}
v(x,t)=\frac 1 {\sqrt{2t}}\, u(y,s),
\end{equation}
for
\begin{equation}\label{variables}
y=\frac x {\sqrt{2t}},\quad s=\log(\sqrt{2t}),
\end{equation}
where $u$ is time-periodic with period $\log(\lambda)$ and solves the \emph{time-dependent Leray equations}
\begin{equation} 
\begin{array}{ll}
\label{eq:timeDependentLeray}
 \partial_s u-\Delta u-u-y\cdot\nabla u +u\cdot \nabla u +\nabla p = 0&\mbox{~in~}\R^3\times \R
\\  \nabla\cdot u = 0&\mbox{~in~}\R^3\times \R.
\end{array}
\end{equation}
Note that the \emph{self-similar transform} \eqref{ansatz2}--\eqref{variables} gives a one-to-one correspondence of solutions of \eqref{eq:NSE} and that of \eqref{eq:timeDependentLeray}. Moreover, when $v_0$ is SS or DSS, the initial condition $v|_{t=0}=v_0$ corresponds to a boundary condition for $u$ at spatial infinity, see Section \ref{sec:DSS}.

The fact that \eqref{eq:stationaryLeray} is time-independent motivates an analogy between the self-similar profile and solutions to the steady state Navier-Stokes equations.  It is known for certain large data and appropriate forcing that solutions to the stationary Navier-Stokes boundary value problem are non-unique \cite{Galdi,Temam}.  In \cite{JiaSverak}, Jia and \v Sver\'ak conjecture that similar non-uniqueness results might hold for solutions to \eqref{eq:stationaryLeray}.  These solutions would necessarily involve large data but, until recently, existence results for self-similar solutions were only known for small data (for small data existence of forward self-similar solutions see \cite{GiMi,Kato,CP,Koch-Tataru}).  Jia and \v Sver\'ak addressed this in \cite{JiaSverak} where they proved the existence of a forward self-similar solution using Leray-Schauder degree theory for \emph{large $-1$-homogeneous} initial data which is \emph{locally H\"older continuous} away from the origin. Similar results were later proven in \cite{Tsai-DSSI} for $\lambda$-DSS solutions with factor \emph{close to one} where closeness is determined by the local H\"older norm of $v_0$ away from the origin.  It is also shown in \cite{Tsai-DSSI} that the closeness condition on $\lambda$ can be eliminated if the initial data is axisymmetric with no swirl. 
In  Korobkov-Tsai \cite{KT-SSHS}, the existence of self-similar solutions on the half space (and the whole space) is established for appropriately smooth initial data.  The approach of \cite{KT-SSHS} differs from \cite{JiaSverak} and \cite{Tsai-DSSI} in that the existence of a solution to the stationary Leray equations \eqref{eq:stationaryLeray} is established directly.  It also gives a second proof of the main result of \cite{JiaSverak}.  {A new approach is necessary in 
\cite{KT-SSHS} due to lack of spatial decay estimates, which gives  \emph{global compactness} needed by  the Leray-Schauder theorem in 
\cite{JiaSverak} and \cite{Tsai-DSSI}.}

The main goal of the present paper is to construct $\lambda$-DSS solutions for any $\lambda>1$ for a very general class of $L^3_w(\R^3)$, possibly rough data. {A key difference between this paper and  \cite{JiaSverak, Tsai-DSSI, KT-SSHS} is the lack of \emph{local compactness}, which is required by the Leray-Schauder theorem and is provided by the regularity theory. In contrast, the regularity of general DSS solutions is not known yet.
}

Since $L^3_w(\R^3)$ embeds continuously into the space of uniformly locally square integrable functions $L^2_{u\,loc}$ it is appropriate to seek \emph{local Leray solutions}. For our purpose, we only consider global in time solutions.

\begin{definition}[Local Leray solutions]\label{def:localLeray} A vector field $v\in L^2_{loc}(\R^3\times [0,\infty))$ is a local Leray solution to \eqref{eq:NSE} with divergence free initial data $v_0\in L^2_{u\,loc}$ if:
\begin{enumerate}
\item for some $\pi\in L^{3/2}_{loc}(\R^3\times [0,\infty))$, the pair $(v,\pi)$ is a distributional solution to \eqref{eq:NSE},
\item for any $R>0$, $v$ satisfies
\begin{equation}\notag
\esssup_{0\leq t<R^2}\,\sup_{x_0\in \R^3}\, \int_{B_R(x_0 )}\frac 1 2 |v(x,t)|^2\,dx + \sup_{x_0\in \R^3}\int_0^{R^2}\int_{B_R(x_0)} |\nabla v(x,t)|^2\,dx \,dt<\infty,\end{equation}
\item for any $R>0$, $v$ satisfies
\begin{equation}\notag 
\lim_{|x_0|\to \infty} \int_0^{R^2}\int_{B_R(x_0 )} | v(x,t)|^2\,dx \,dt=0,
\end{equation}
\item for all compact subsets $K$ of $\R^3$ we have $v(t)\to v_0$ in $L^2(K)$ as $t\to 0^+$,
\item $v$ is suitable in the sense of Caffarelli-Kohn-Nirenberg, i.e., for all cylinders $Q$ compactly supported in  $ \R^3\times(0,\infty )$ and all non-negative $\phi\in C_0^\infty (Q)$, we have 
\begin{equation}\label{eq:localEnergyNSE}
2\int \int |\nabla v|^2\phi\,dx\,dt \leq \int\int |v|^2(\partial_t \phi + \Delta\phi )\,dx\,dt +\int\int (|v|^2+2\pi)(v\cdot \nabla\phi)\,dx\,dt.
\end{equation}
\end{enumerate}
\end{definition}

The concept of local Leray solutions was introduced by  Lemari\`e-Rieusset \cite{LR}, where he showed the existence of global in time  local Leray solutions if $v_0$ further belongs to $E_2$,  the closure of $C_0^\infty$ in the $L^2_{u\,loc}(\R^3)$ norm.  See Kikuchi-Seregin \cite{KiSe} for more details. In particular, condition 3 justifies a formula of the pressure $\pi$ in terms of the velocity $v$, see \cite[(1.9)]{KiSe} and
 \cite[(3.3)]{JiaSverak}.

The definition of suitability appearing above is taken from \cite{JiaSverak} and \cite{LR}.  It defines the local energy estimate in terms of test functions compactly supported away from $t=0$.  This is an unnecessary restriction.  In particular, conditions 4 and 5 from Definition \ref{def:localLeray} together imply the local energy inequality is valid for test functions with support extending down to $t=0$. In this case $\int  |v(0)|^2\phi\,dx$ should be added to the right hand side of \eqref{eq:localEnergyNSE}.

Let $e^{t\Delta}v_0(x)=\int_{\R^3} (4\pi t)^{-3/2}e^{-|x-z|^2/t}v_0(z)\,dz$; this is the solution to the homogeneous heat equation in $\R^3$. Our main objective is to prove the following theorem.

\begin{theorem}\label{thrm:main}
Let $v_0$ be a divergence free, $\lambda$-DSS vector field for some $\lambda >1$ and satisfy 
\begin{equation}\label{ineq:decayingdata}
\|v_0\|_{L^3_w(\R^3)}\leq c_0,
\end{equation} for a possibly large constant $c_0$. Then, there exists a local Leray solution $v$ to \eqref{eq:NSE} which is $\lambda$-DSS and additionally satisfies
\begin{equation}\notag
 \|  v(t)-e^{t\Delta}v_0 \|_{L^2(\R^3)}\leq C_0\,t^{1/4}
\end{equation}
for any $t\in (0,\infty)$ and a constant $C_0=C_0(v_0)$.
\end{theorem}

\noindent Comments on Theorem \ref{thrm:main}:

\begin{enumerate} %
\item The constant $c_0$ is allowed to be large. The condition $v_0\in L^3_w(\R^3)$ is weaker than initial conditions found in previous constructions for large data.  In particular, \cite{JiaSverak,Tsai-DSSI} require $v_0$ to be in $C^\alpha_{loc}(\R^3\setminus \{  0\})$ for some $\alpha>0$.  Additionally, in contrast to \cite{Tsai-DSSI}, our construction does not restrict the size of $\lambda$. 

\item When $v$ is strictly DSS, smoothness is not known \emph{a priori}. In contrast, the DSS solutions constructed in \cite{Tsai-DSSI} are smooth.  This is a consequence of the fact that, whenever $\lambda-1$ is sufficiently small, a local regularity theory is available for $\lambda$-DSS solutions in the local Leray class.  One may say that \cite{Tsai-DSSI} constructs strong DSS solutions for special (large) initial data
while this paper considers weak DSS solutions for general initial data.

\item When $v_0$ is $\la$-DSS, it will be shown in Lemma \ref{lemma:equivalence} that \eqref{ineq:decayingdata} is equivalent to $v_0 \in L^3(A_1)$ where $A_1 = \{ x \in \R^3: 1\le |x|< \la\}$. In fact, Theorem \ref{thrm:main} is true under a weaker condition on $v_0$ than \eqref{ineq:decayingdata}. Recall that
the Morrey space $M^{p,\al}= M^{p,\al}(\R^3)$ is the collection of functions $f$ such that $\norm{f}_{M^{p,\al}} :=\sup_{x \in \R^3, r>0} \bket{ r^{-\al} \int_{B(x,r)} |f|^p}^{1/p} < \I$. 
Theorem \ref{thrm:main} remains valid if we replace \eqref{ineq:decayingdata} by
\begin{equation}\label{data.Morrey}
v_0 \in L^2(A_1), \quad \limsup_{r \to 0_+} \sup_{x \in A_1} r^{-1} \int_{B(x,r)} |v_0|^2 \le \e_0,
\end{equation}
for some constant $\e_0>0$ sufficiently small.  Such $v_0$ is in $M^{2,1}$ since it is $\la$-DSS.
Condition \eqref{data.Morrey}
with $\e_0=0$ is all we need to prove Lemma \ref{th:2.1} on Assumption \ref{AU_0} on the profile $U_0$, and our construction actually only needs a weakened Assumption \ref{AU_0} ($\lim_{R \to \I} \Theta(R) \le \e_0$) allowing small $\e_0>0$. For the convergence of the solution $v(t)$ to initial value $v_0$ as $t \to 0$,
although $e^{t \Delta}$ is not a $C_0$-semigroup on the Morrey space $ M^{2,1} $
as noted in Kato \cite[Remark 2.3]{Kato}, it is a $C_0$-semigroup on weighted $L^2$ spaces  $L^2_{-k/2}=\{ f: \int _{\R^3}\frac {|f(x)|^2}{(1+|x|)^k} dx < \I \}$,  which  
 $M^{2,1}$ is imbedded into if $k>1$. Thus $e^{t \Delta}v_0 \to v_0$ as $t \to 0$ in the norm of $L^{2}_{-k/2}$, which implies local $L^2$-convergence.

\item
An example of initial data in $M^{2,1}$ that does not satsfy \eqref{data.Morrey} is the following.
Denote by $\chi_S$ the characteristic function for the set $S$. Fix $x_0=(1,0,0)$ and let $x_k=2^kx_0$ and $B_k=B_{2^{k-2}}(x_k)$ denote the ball in $\R^3$ centered at $x_k$ of radius $2^{k-2}$.  Let
\begin{equation}
u=\sum_{k\in \mathbb Z}u_k,\quad
\text{where}
\quad
u_k(x)= \frac {\chi_{B_k}(x)} {|x-x_k|}.
\end{equation}
For $\lambda=2$ and all $k\in \mathbb Z$ we have $\lambda u_k(\lambda x)=u_{k-1}(x)$ and it follows that $u$ is $\lambda$-DSS.  This function belongs to $M^{p,3-p}(\R^3)\setminus L^3_w{(\R^3)}$ for $1\leq p<3$ (these are the critical Morrey spaces).  Our approach breaks down for data like this (unless we multiply it by a small number) because we are unable to control the spatial decay of $e^{t\Delta}v_0$. If $v_0\in L^3_w$ on the other hand, discrete self-similarity implies some spatial decay -- this will be made clear in Lemma \ref{lemma:V0decay}.  

\item  For the usual Leray-Hopf weak solutions, it is well known that the hypothetical singular set is contained in a  compact subset of space-time.  We would call this property {\it eventual regularity}. The eventual regularity of local Leray solutions is unclear:
If a $\la$-DSS solution $v$ is singular at some point $(x_0,t_0)$, it is also singular at $(\la^k x_0,\la^{2k} t_0)$ for all integers $k$. 
Since $v$ is regular if it is SS or if $\la$ is close to one, the possibility of a non-compact singular set for some local Leray solutions is suggested by
Theorem \ref{thrm:main}, but not by \cite{JiaSverak,Tsai-DSSI}.

\end{enumerate} %

Our approach is similar to \cite{KT-SSHS} in that we prove a priori estimates for  the Leray equations \eqref{eq:stationaryLeray} and
directly prove the existence of the ansatz \eqref{ansatz2} (in \cite{KT-SSHS} this is done for a half space version of \eqref{ansatz1}).  In contrast, the solutions of \cite{JiaSverak} and \cite{Tsai-DSSI} are constructed using the Leray-Schauder fixed point theorem
for the equation for $\tilde v = v-e^{t\Delta}v_0$, namely,
\begin{equation}\notag
 \tilde v = K(\tilde v) :=T(F(e^{t\Delta}v_0+\tilde v)), \quad K:X \to X,
\end{equation}
where $T$ is essentially the Stokes solver, i.e., the solution operator of the time-dependent Stokes system, $F(u)=u\otimes u$ is the nonlinearity, and $X$ is some function class for $\tilde v$. In this approach one needs to show the compactness of $K(\tilde v)$ which involves the spatially local and asymptotic properties of $K(\tilde v)$. When the norm of $X$ is subcritical or critical (e.g. Prodi-Serrin class), the local compactness is provided by the regularity theory. When the norm of $X$ is {\em supercritical} (e.g.~energy class), the usual bootstrap argument does not provide better regularity.
One may hope to gain the local compactness using the local energy inequality, but this is nonlinear and not well-defined for the Stokes solver $T$. Hence, local energy inequality is likely not available for $K(\tilde v)$ even if it holds for $\tilde v$.

The difference between  \cite{KT-SSHS} and this paper is that \cite{KT-SSHS} proves the a priori bounds for \eqref{eq:stationaryLeray} 
by a contradiction argument and a study of the Euler equations, and constructs the solutions by the method of invading domains, while the current paper proves the bound directly with a computable constant, and constructs the solutions in the whole space directly. 

The key observation leading to the explicit a priori bound is the following: If a solution $u(y,s)$ of the Leray equations \eqref{eq:timeDependentLeray}
asymptotically agrees with a given $U_0(y,s)$, (e.g.~$u-U_0 \in L^\I(\R,L^2(\R^3))$), then the difference $U=u-U_0$  formally satisfies 
\begin{equation}\label{eq1.14}
 \int_0^T\! \int \bke{ |\nb U|^2 + \frac 12 |U|^2 }= \int_0^T \!\int (U \cdot \nb) U \cdot U_0 - \int_0^T\! \int \mathcal R(U_0)\cdot U ,
\end{equation}
where the source term $\mathcal R(U_0)$ will be given in \eqref{RW.def}.
The integral {$\iint (U\cdot \nb) U \cdot U_0$} is usually out of control for large $U_0$, but now we have $\iint |V|^2$ on the left side (which is not available for Navier-Stokes), and $U_0$ decays. Thus this trouble term can be controlled if the local part of $U_0$ is suitably ``cut off.'' See Section \ref{sec2} for details.

\medskip
If $v_0$ is SS, we can use our result to construct a SS solution by considering the sequence of solutions obtained from Theorem \ref{thrm:main} treating the data as $\lambda_k$-DSS for an appropriate sequence $\lambda_k$ which decreases to $1$ as $k\to\infty$. Alternatively, one can also construct the SS solution directly without involving the time dependence.
Indeed we have the following theorem.

\begin{theorem}
\label{thrm:selfsimilardata}
Let $v_0$ be a $(-1)$-homogeneous divergence free vector field in $\R^3$ which satisfies \eqref{ineq:decayingdata}
for a possibly large constant $c_0$. Then, there exists a local Leray solution $v$ to \eqref{eq:NSE} which is self-similar and additionally satisfies
\begin{equation}\notag
 \|  v(t)-e^{t\Delta}v_0 \|_{L^2(\R^3)}\leq C_0\,t^{1/4}
\end{equation}
for any $t\in (0,\infty)$ and a constant $C_0=C_0(v_0)$.
\end{theorem}

This solution is infinitely smooth, as are the solutions from \cite{JiaSverak}, see e.g.~\cite{Grujic}.

As mentioned earlier, the first construction of large self-similar solutions is given in  \cite{JiaSverak} using local H\"older estimates and the Leray-Schauder theorem. The second construction is given in \cite{KT-SSHS} using an a priori bound for Leray equations derived by a contradiction argument and a study of the Euler equations. The current paper provides a new (third) construction based on the explicit a priori bound.

We expect our method could give an alternative construction of self-similar solutions in the \emph{half space} $\R^3_+$, after the first one in \cite{KT-SSHS}:
Assuming the decay estimates for $e^{-A}v_0$ of  \cite{KT-SSHS}, one could get a priori bounds by suitable cut-off (this requires some work) and avoid the contradiction argument.
This solution is only in distributional sense and the pressure is not defined. %

The remainder of this paper is organized as follows. In Section 2 we construct solutions to the time periodic Leray system on $\R^3\times \R$ which satisfy a local energy estimate.  In Section 3 we use these solutions to recover a discretely self-similar local Leray solution on $\R^3\times (0,\infty)$, thereby proving Theorem \ref{thrm:main}.  Finally, in Section 4 we give two proofs of Theorem \ref{thrm:selfsimilardata}.  One is essentially a corollary of Theorem \ref{thrm:main} while the other constructs a stationary weak solution to Leray's equation directly.

\medskip
\emph{Notation.}\quad We will use the following function spaces:
\begin{align*}
&\mathcal V=\{f\in C_0^\infty({ \R^3;\R^3}) ,\, \nabla \cdot f=0 \},
\\& X = \mbox{the closure of~$\mathcal V$~in~$H_0^1(\R^3)$} ,
\\& H = \mbox{the closure of~$\mathcal V$~in~$L^2(\R^3)$},
\end{align*}where $H_0^1(\R^3)$ is the closure of $C_0^\infty(\R^3)$ in the Sobolev space $H^1(\R^3)$.  Let $X^*(\R^3)$ denote the dual space of $X(\R^3)$. 
Let $(\cdot,\cdot)$ be the $L^2(\R^3)$ inner product and $\langle\cdot,\cdot\rangle$ be the dual product for $H^1$ and its dual space $H^{-1}$, or that for $X$ and $X^*$.
Denote by  $\mathcal D_T$ the collection of all smooth divergence free vector fields in $\R^3 \times \R$ which 
are time periodic with period $T$ and whose supports are compact in space.

\section{The time periodic Leray system}\label{sec2}

In this section we construct a periodic weak solution to the Leray system
\begin{equation}
\label{eq:wholeSpaceLeray}
\begin{array}{ll}
	\partial_s u -\Delta u=u+y\cdot \nabla u -\nabla p -u\cdot\nabla u	&\mbox{~in~}\R^3\times \R
	\\  \nabla\cdot u = 0  &\mbox{~in~}\R^3\times \R
	\\ 	\displaystyle \lim_{|y_0|\to\infty} \int_{B_1(y_0)}|u(y,s)-U_0(y,s)|^2\,dx= 0& \mbox{~for all~}s\in \R
	\\  u(\cdot,s)=u(\cdot, s+T) &\mbox{~in~}\R^3\mbox{~for all~}s\in \R,
\end{array}
\end{equation}
for a given $T$-periodic divergence free vector field $U_0$.  Here, $U_0$ serves as the boundary value of the system and is required to satisfy the following assumption.

\begin{assumption} \label{AU_0}
The vector field $U_0(y,s) :\R^3 \times \R \to \R^3$ is continuously differentiable in $y$ and $s$,  periodic in $s$ with period $T>0$, divergence free, and satisfies
\begin{align*}
& 	\partial_s U_0-\Delta U_0-U_0-y\cdot \nabla U_0 = 0, 
\\& U_0\in L^\infty (0,T;L^4\cap L^q(\R^3)),
\\& \partial_s U_0\in L^\infty(0,T;L_{loc}^{6/5}(\R^3)),
\end{align*}
and
\[
\sup_{s\in [0,T]}\|U_0  \|_{L^q(\R^3\setminus B_R)}\leq \Theta(R),
\]
for some $q\in (3,\infty]$ and $\Theta:[0,\infty)\to [0,\infty)$ such that $\Theta(R)\to 0$ as $R\to\infty$.
\end{assumption}

Note that membership in $C^1$ guarantees that $\partial_s U_0\in L^\infty(0,T;L_{loc}^{6/5}(\R^3))$ and we only mention this inclusion explicitly since later estimates will depend on the quantity $\norm{U_0}_{L^\infty(0,T;L_{loc}^{6/5}(\R^3))}$.

For a given $W(y,s)$ and any $\zeta \in C^1_0(\R^3)$, let
\begin{equation}\label{LW.def}\notag
LW = 	\partial_s W-\Delta W-W-y\cdot \nabla W ,
\end{equation}
and
\begin{equation}
\bka{LW,\zeta} =(\partial_s W-W-y\cdot \nabla W,\zeta) + (\nabla W, \nabla \zeta).
\end{equation}

Periodic weak solutions to \eqref{eq:wholeSpaceLeray} %
are defined as follows.

\begin{definition}[Periodic weak solution]
\label{def:periodicweaksolutionR3} 
Let $U_0$ satisfy Assumption \ref{AU_0}. 
The field $u$ is a periodic weak solution to \eqref{eq:wholeSpaceLeray} if it is divergence free, if 
\begin{equation}\notag
U:= 
u-U_0\in L^\infty(0,T;L^2(\R^3))\cap L^2(0,T;H^1(\R^3)),
 \end{equation} 
and if
\begin{equation}\label{u.eq-weak}
\int_0^T \big( (u,\partial_s f)-(\nabla u,\nabla f)+(u+y\cdot\nabla u-u\cdot\nabla u ,f)  \big)  \,ds =0,
\end{equation} 
holds for all $f \in \mathcal D_T$.  
This latter condition implies that $u(0)=u(T)$.

\end{definition}

If $u$ satisfies this definition then there exists a pressure $p$ so that $(u,p)$ constitute a distributional solution to \eqref{eq:wholeSpaceLeray} (see \cite{Temam}; we will provide more details in our proof).  In the SS variables our notion of suitability mirrors that in the physical variables.
\begin{definition} [Suitable periodic weak solution]
Let $U_0$ satisfy Assumption \ref{AU_0}. 
A pair $(u,p)$ is a \emph{suitable periodic weak solution} to \eqref{eq:wholeSpaceLeray} 
on $\R^3$ if both are time periodic with period $T$,
 $u$ is a periodic weak solution on $\R^3$, $p\in L^{3/2}_{loc}(\R^4)$, the pair $(u,p)$ solves \eqref{eq:wholeSpaceLeray} in the sense of distributions, and the local energy inequality holds:%
\begin{align}\label{ineq:localEnergy}
 \int_{\R^4}\bigg( \frac12 {|  u|^2}   +|\nabla   u|^2 \bigg)\psi\,dy\,ds &\leq \int_{\R^4} \frac {|   u|^2} 2 \big(\partial_s \psi +\Delta \psi   \big)\,dy\,ds
\\&\quad +  \int_{\R^4} \bigg( \frac 1 2  |   u|^2 ((  u- y)\cdot \nabla \psi ) +   p (  u\cdot \nabla \psi) \bigg)\,dy\,ds,\notag 
\end{align}
for all nonnegative $\psi \in C_0^\infty(\R^4)$.
\end{definition}

The main result of this section concerns the existence of suitable periodic weak solutions. 

\begin{theorem}[Existence of suitable periodic weak solutions to \eqref{eq:wholeSpaceLeray}]\label{thrm:existenceOnR3}
Assume $U_0(y,s)$ satisfies Assumption \ref{AU_0} with $q=10/3$. Then \eqref{eq:wholeSpaceLeray} has a periodic suitable weak solution $(u,p)$ in $\R^4$ with period $T$.
\end{theorem}  

We need $3<q\le \frac{18}5$ to show $p \in L^{3/2}_{x,t,loc}$, and it is convenient to take $q=10/3$.

Ideally we would prove the existence of a divergence free time-periodic vector field $U$ where $u=U+U_0$ and $U$ satisfies a perturbed version of \eqref{eq:wholeSpaceLeray}. In view of \eqref{eq1.14}, doing so would require the constant from the pointwise bound on $U_0(y,s)$ be small to ensure that
\begin{equation}
   \int (f\cdot\nabla f)\cdot U_0 \,dy \leq \alpha ||f||_{H_0^1(\R^3)}^2,
\end{equation}
for any $f\in H_0^1(\R^3)$ and a small constant $\alpha$.  To get around this issue we replace $U_0$ by a perturbation $W$ which eliminates the possibly large behavior of $U_0$ near the origin. Fix $Z\in C^\infty(\R^3)$ with $0 \le Z \le 1$, $Z(x)=1$ for $|x|>2$ and $Z(x)=0$ for $|x|<1$. This can be done so that $|\nb Z|+|\nb^2 Z| \lec 1$.
 For a given $R>0$, let $\xi_{R}(y)=Z(\frac yR)$. It follows that $|\nabla^k \xi_R|\lesssim R^{-k}$ for $k\in \{ 0,1\}$.

\begin{lemma}[Revised asymptotic profile]
\label{lemma:W}
Fix $q\in (3,\infty]$ and suppose $U_0$ satisfies Assumption \ref{AU_0} for this $q$. 
Let $Z\in C^\infty(\R^3)$ be as above.
For any $\alpha\in (0,1)$, there exists $R_0=R_0(U_0,\alpha)\ge 1$ so that letting $\xi(y) =Z(\frac y{R_0})$ and setting
\begin{equation}
  W (y,s)= \xi(y) U_0(y,s) + w(y,s),
\end{equation}
where 
\begin{equation}
w(y,s)=\int_{\R^3}\nabla_y \frac 1 {4\pi |y-z|} \nabla_z \xi(z) \cdot U_0 (z,s) \,dz,
\end{equation}
we have that $W$ is locally continuously differentiable in $y$ and $s$, $T$-periodic, divergence free,
 $U_0 - W \in L^\infty(0,T; L^2(\R^3))$, and 
\begin{equation}\label{ineq:Wsmall}
\|W\|_{L^\infty(0,T;L^q(\R^3))}\leq \alpha, %
\end{equation} 
\begin{equation}\label{WL4.est}
\norm{W}_{L^\infty(0,T;L^4(\R^3))}\leq c(R_0,U_0),
\end{equation}
and
\begin{equation}
\label{LW.est}
\norm{LW}_{L^\infty(0,T; H^{-1}(\R^3))} \leq c(R_0,U_0), %
\end{equation}
where $LW$ is given in \eqref{LW.def}, $c(R_0,U_0)$ depends on $R_0$ and quantities associated with $U_0$ which are finite by Assumption \ref{AU_0}.
\end{lemma}

{\it Remark.} The correction term $w$, introduced to make $\div W=0$,  usually has compact support, see \cite[III.3]{Galdi}. Similar non-compact corrections have also been used, e.g.~in \cite{KMT2012, LuoTsai}.

\begin{proof}
We will typically suppress the $s$ dependence.
Since $U_0$ is divergence free and $w = \nabla (-\Delta)^{-1}  ( \nabla \xi \cdot U_0)$, we have $\div W = \nabla \xi \cdot U_0 + \div w = 0$. 

We first prove the bound \eqref{ineq:Wsmall}. Since $U_0$ is divergence free we obtain using the integral formula for $w$ and the Calderon-Zygmund theory that
\[
\|w\|_{L^q(\R^3)}\leq c_q \|\xi U_0  \|_{L^q(\R^3)}\leq c_q\Theta(R_0),
\]
where $\Theta$ is given by Assumption \ref{AU_0} and $c_q$ depends on $q$.  Then, assuming $R_0$ is large enough that $\Theta(R_0)\leq \alpha (1+c_q)^{-1}$, it follows that
\[
\|W\|_{L^q(\R^3)}\leq (1+c_q)\Theta(R_0)\leq \alpha,
\]
which proves \eqref{ineq:Wsmall}.

The second inequality \eqref{WL4.est} follows immediately from the Calderon-Zygmund theory and Assumption \ref{AU_0}.

Estimates for the third inequality \eqref{LW.est} are more involved. Note that $LW=L(\xi U_0)+Lw$.  Using the definition of $w$ we have
\[ \partial_s w(y,s)=\nabla_y\int_{\R^3} \frac 1 {4\pi |y-z|} \nabla_z \xi(z) \cdot \partial_s U_0 (z,s) \,dz,\]
and the Hardy-Littlewood-Sobolev inequality implies that 
\EQ{\label{pdsw.est}
\|\partial_s w\|_{L^2(\R^3)}&= \bigg\| \nabla_y \int_{\R^3} \frac 1 {4\pi |y-z|} \nabla_z \xi(z) \cdot \partial_s U_0 (z,s) \,dz \bigg\|_{L^2(\R^3)}
\\& \leq c \|\nabla \xi \cdot \partial_s U_0\|_{L^{6/5}(\R^3)},
}
which is finite by Assumption \ref{AU_0}.
We have also that
\begin{equation}
\label{eq2.13}
| w(y)|\lesssim  \int _{R_0 < |z|<2R_0} \frac1{|y-z|^2} \frac {|U_0(z)|} {R_0} dz \lec
\left \{
\begin{split} 
& R_0^{-3/4}\|U_0 \|_{L^4(\R^3)} \quad &\text{if } |y|\le 4R_0
\\
&  |y|^{-2}R_0^{5/4}\|U_0 \|_{L^4(\R^3)} \quad &\text{if } |y|>4R_0
\end{split}
\right . .
\end{equation}
These estimates are independent of time and therefore 
\EQ{ \label{w.est}
\| w\|_{L^\infty(0,T;L^2(\R^3))}\leq C(R_0,U_0).
}

We next show
\begin{equation}
\label{ineq:wgradient}|\nabla w(y)|\leq \frac {C(R_0,U_0)} {1+|y|^{3}},
\end{equation}which will allow us to conclude our estimate for $\|Lw \|_{L^2}$.
If $|y|\leq 4R_0$ we have
\begin{equation}
\nabla w(y) =\int_{\R^3}\nabla_y \frac 1 {4\pi |y-z|} \nabla_z(\nabla_z \xi(z) \cdot U_0 (z)) \,dz,
\end{equation}
and, since $U_0$ is continuously differentiable we have $||\nabla U_0||_{L^\infty (B_{2R_0})}<\infty $ and thus
\begin{equation}
|| \nabla w(y)||_{L^\infty (B_{4R_0})}\leq C(R_0,U_0),
\end{equation}
where $c_1$ depends on $R_0$ and $||\nabla U_0||_{L^\infty (B_{2R_0})}$.
If $|y|\geq 4R_0$ then
\begin{equation}
\nabla w(y) = \int \nabla_z \nabla_y \frac 1{4\pi |z-y|}  (\nabla \xi \cdot U_0)(z)\,dz,
\end{equation}
 and it follows that
\begin{equation}
|\nabla w(y)|\leq \frac {c R_0^{5/4} \|U_0\|_{L^4(\R^3)}} {|y|^3}.
\end{equation}
Thus we have \eqref{ineq:wgradient}.
The estimates
\eqref{pdsw.est}, \eqref{w.est}, and \eqref{ineq:wgradient} show that $Lw\in L^\infty(0,T;H^{-1}(\R^3))$.

We now focus on $L(\xi U_0)$.  Note that, because $LU_0=0$ by Assumption \ref{AU_0},
\EQ{
L(\xi U_0) &= \xi L U_0 + W_2 = W_2,}
where \EQ{
W_2&=-(\Delta \xi) U_0 -2(\nabla \xi \cdot \nabla) U_0 -(y\cdot \nabla \xi) U_0.
}
Since  both $U_0$ and $\nabla U_0$ belong to $L^\infty_{loc}(\R^3\times \R)$ and $\nabla \xi$ is compactly supported, 
\[
\sup_{0\leq s\leq T}\norm{W_2(\cdot,s)}_{L^1\cap L^\infty(\R^3)} \le C(R_0,c_0).
\]

The above estimates show \eqref{LW.est} and complete the proof.
\end{proof}



To solve for $u$, we will decompose $u=W+ U$, where $W$ is as in Lemma \ref{lemma:W} for $\alpha = 1/4$, and hence we can drop the $R_0$ dependence in $C(U_0)$.
Note $U$ satisfies a perturbed Leray system, namely
\begin{equation} \label{perturbed-Leray}
L U + (W+U)\cdot \nabla U + U\cdot \nabla W +\nabla p= -  \mathcal R(W), \quad \div U=0,
\end{equation} 
where the source term is
\begin{equation}\label{RW.def}
\mathcal{R}(W) := 	\partial_s W-\Delta W-W-y\cdot \nabla W + W\cdot\nabla W.
\end{equation}

To obtain suitable weak solutions (as opposed to just weak solutions) to \eqref{eq:wholeSpaceLeray}, we first construct smooth solutions to the mollified version of \eqref{perturbed-Leray}, see e.g.~discussions in \cite{BCI}.  For all $\epsilon>0$, let $\eta_\epsilon(y)=\epsilon^{-3}\eta(y/\epsilon)$ for some $\eta\in C_0^\infty$ satisfying $\int_{\R^3}\eta\,dy=1$. 
We seek a solution of the form $u_\e=U_\e+W$ where $U_\e$ is %
$T$-periodic, decays faster than $W$ at spatial infinity, and satisfies the {\em mollified perturbed Leray equations} for $U=U_\e$ and $p=p_\e$,
\begin{align}\label{eq:mollifiedLeray}
L U + (W+\eta_\e* U)\cdot \nabla U + U\cdot \nabla W +\nabla p= -  \mathcal R(W), \quad \div U=0,
\end{align}
on $\R^3\times [0,T]$. 
 The weak formulation of \eqref{eq:mollifiedLeray} is 
\begin{align}\label{eq:boundedWeakForm}
\frac d {ds}(U,f)
&=
		- (\nabla U,\nabla f) 
		+ (U+y\cdot \nabla U, f)
		- ((\eta_\epsilon *U) \cdot\nabla U, f)
\\\notag&\quad
		-(W\cdot\nabla U+U\cdot \nabla W,f)-\langle \mathcal{R}(W),f\rangle,
\end{align}
and holds for all $f \in \mathcal D_T$ and a.e.~$s\in (0,T)$.

We use the Galerkin method following \cite{GS06} (see also $\cite{Morimoto,Temam}$).
Let $\{a_{k}\}_{k\in \N}\subset \mathcal V$ be an orthonormal basis of $H$.
For a fixed $k$, we look for an approximation solution of the form $U_k(y,s)= \sum_{i=1}^k b_{ki}(s)a_i(y)$.
We first prove the existence of and \emph{a priori} bounds for $T$-periodic solutions $b_k=(b_{k1},\ldots,b_{kk})$ to the system of ODEs
\begin{align}\label{eq:ODE}
\frac d {ds} b_{kj} = & \sum_{i=1}^k A_{ij}b_{ki} +\sum_{i,l=1}^k B_{ilj} b_{ki}b_{kl} +C_j,%
\end{align}
for $j\in \{1,\ldots,k\}$,
where
\begin{align}
\notag A_{ij}&=- (\nabla a_{i},\nabla a_j) 
		+ (a_i+y\cdot \nabla a_i, a_j) 
		 -(    a_i\cdot \nabla W,a_j)
		- (W\cdot\nabla a_i, a_j)
\\\notag B_{ilj}&=- (\eta_\epsilon *a_i \cdot\nabla a_l, a_j)
\\\notag C_j&=-\langle \mathcal R (W),a_j\rangle.
\end{align}
 
\begin{lemma}[Construction of Galerkin approximations]\label{lemma:Galerkin} Fix $T>0$ and let $W$ satisfy the conclusions of Lemma \ref{lemma:W} with $\alpha=\frac 14$.
\begin{enumerate}
\item For any $k\in \mathbb N$ and $\epsilon>0$, the system of ODEs \eqref{eq:ODE} has a $T$-periodic solution $b_{k}\in H^1(0,T)$.
\item Letting
\begin{equation} \notag
U_k(y,s)=\sum_{i=1}^k b_{ki}(s)a_i(y),
\end{equation}we have 
\begin{equation}\label{ineq:uniformink}
||U_k||_{L^\infty (0,T;L^2(\R^3))} + ||U_k||_{L^2(0,T;H^1(\R^3))}<C,
\end{equation}where $C$ is independent of both $\epsilon$ and $k$.
\end{enumerate}
\end{lemma}
\begin{proof}
Our argument is standard (see \cite{GS06,Temam}).
Fix $k\in \N$. For any  $U^{0}\in \operatorname{span}(a_1,\ldots,a_k)$, 
there exist $b_{kj}(s)$ uniquely solving \eqref{eq:ODE} with initial value $b_{kj}(0)=(U^{0},a_j)$, and belonging to $H^1(0,\tilde T)$ for some time $0<\tilde T\leq T$.  If $\tilde T<T$ assume it is maximal -- i.e.~$||b_{k}(s)||_{L^2}\to\infty$ as $s\to \tilde T^-$.  By multiplying the $j$-th equation of \eqref{eq:ODE} by $b_{kj}$ and summing, since certain cubic terms vanish, we obtain 
\begin{equation} \label{ineq:1}
\frac 1 2 \frac d {ds} ||U_k||_{L^2}^2 + \frac 1 2 ||U_k||_{L^2}^2+ ||\nabla U_k||_{L^2}^2\leq - ( U_k\cdot \nabla W, U_k ) - \langle \mathcal{R}(W), U_k\rangle. 
\end{equation}
Note that \eqref{ineq:Wsmall} and the fact that $U_k$ is divergence free guarantee that
\begin{equation}
  \big| ( U_k\cdot \nabla W, U_k )  \big| \leq  \frac 1 8 ||U_k||_{H^1}^2 .%
\end{equation}

By $\cR(W)= LW + \div(W\otimes W)$, and \eqref{ineq:Wsmall},
\begin{equation} \label{ineq:2}
|(\cR(W),U_k)| \le (\norm{LW}_{H^{-1}} +   \|W  \|_{L^4}^2 ) \norm{U_k}_{H^1}  \le C_2+ \frac 1 8 ||U_k||_{H^1}^2 .%
\end{equation}
where $C_{2}=C(\norm{LW}_{H^{-1}} +  \|W  \|_{L^4}^2)^2 $ is independent of $s$, $T$, $k$, and $\epsilon$.

Using Lemma \ref{lemma:W}, the estimates \eqref{ineq:1}--\eqref{ineq:2} imply
\begin{equation}  \label{ineq:kenergyevolution}
	 \frac d {ds} ||U_k||_{L^2}^2
	 +   \frac 1 2 ||U_k||_{L^2}^2
	 +   \frac 1 2 ||\nabla U_k||_{L^2}^2 \leq C_{2} .
\end{equation}
The Gronwall inequality implies
\begin{equation} \label{ineq:gronwall}
\begin{split}
e^{s/2} ||U_k(s)||_{L^2}^2
&\leq ||U^{0}||_{L^2}^2 +  \int_0^{\tilde T} e^{\tau/2}  C_2 \,dt
\\
& \le  ||U^{0}||_{L^2}^2 + e^{T/2} C_2 T
\end{split}
\end{equation}
for all $s\in [0,\tilde T]$. Since the right hand side is finite, $\tilde T$ is not a blow-up time and we conclude that $\tilde T=T$.  

By \eqref{ineq:gronwall} we can choose $\rho>0$ (independent of $k$ \footnote{For usual Navier-Stokes we expect $\rho$ to depend on $k$ since we use the imbedding $H^1_0(B) \to L^2(B)$ for $B\subset \R^k$. But here for Leray system we don't need it.}) so that 
\begin{equation}\notag
 ||U^{0}||_{L^2}\leq \rho \Rightarrow ||U_{k}(T)||_{L^2}\leq \rho.
\end{equation}
Let $T: B_\rho^k\to B_\rho^k$ map $b_{k}(0)\to b_k(T)$, where $ B_\rho^k$ is the closed ball of radius $\rho$ in $\R^k$.  This map is continuous and thus has a fixed point by the Brouwer fixed-point theorem, implying there exists some $U^{0}\in \operatorname{span}(a_1,\ldots,a_k)$ so that $b_k(0)=b_k(T)$. 

It remains to check that \eqref{ineq:uniformink} holds. The $L^\infty L^2$ bound now follows from \eqref{ineq:gronwall}
since $\norm{U^0}_{L^2} \le \rho$, which is independent of $k$ and $\epsilon$.  
 Integrating  \eqref{ineq:kenergyevolution} in $s \in [0,T]$ and using $U_k(0)=U_k(T)$, we get
\begin{equation} \label{eq2.33}
 \frac 1 2 \int_0^T \big(||U_k||_{L^2}^2
+  ||\nabla U_k||_{L^2}^2 \big)\,dt \le C_2 T
\end{equation}
which gives an upper bound for $\| U_k  \|_{L^2(0,T;H^1 )}$ uniform in $k$ and $\epsilon$.
\end{proof}

We are now ready to prove Theorem \ref{thrm:existenceOnR3}.  

\begin{proof}[Proof of Theorem \ref{thrm:existenceOnR3}] 
The Galerkin approximates to the mollified system lead to a solution $U_\epsilon$ through a standard limiting process.  Indeed, under the assumptions of Theorem \ref{thrm:existenceOnR3}, standard arguments (e.g.~those in \cite{Temam}) imply that, for $T>0$ and for any $\epsilon>0$, there exists $T$-periodic $U\in {L^2(0,T;H_0^1(\R^3))}$ (with norm bounded independently of $\epsilon$) and a subsequence of $\{U_k\}$ (still denoted by $U_k$) so that 
\begin{align*}
& U_k\rightarrow U_\epsilon \mbox{~weakly in}~L^2(0,T;X),
\\& U_k\rightarrow U_\epsilon \mbox{~strongly in}~L^2(0,T;L^2(K))  \mbox{~for all compact sets~}K\subset \R^3,
\\& U_k(s)\rightarrow U_\epsilon(s) \mbox{~weakly in}~L^2 \mbox{~for all}~s\in [0,T].
\end{align*}
The weak convergence guarantees that $U_\epsilon(0)=U_\epsilon(T)$.  The limit $U_\e$ is a periodic weak solution of the mollified perturbed Leray system \eqref{eq:mollifiedLeray}.

At this stage we construct a pressure $p_\epsilon$ associated to $U_\e$ for the system \eqref{eq:mollifiedLeray}. This will allow us to obtain a suitable weak solution of  \eqref{eq:wholeSpaceLeray} when we let $\epsilon\to 0$.  Note that $p_\e$ is defined as a distribution whenever $U_\e$ is a weak solution (see \cite{Temam}), but we need to show that $p_\e \in L^{3/2}_{x,t,loc}$ with a bound uniform in $\e$. 
Note that $\div  L(W)=0$ because $W$ is divergence free and, therefore, taking the divergence of \eqref{eq:mollifiedLeray},
 \begin{equation}\label{p.eq}
 -\Delta p_\epsilon =\sum_{i,j}\pd_i \pd_j
 \bkt{(\eta_\epsilon * U _i )U_j + W_i U_j + U_i W_j + W_i W_j}.
  \end{equation}
  Let 
\begin{equation} 
\tilde p_\epsilon =\sum_{i,j}  R_i R_j \bkt{(\eta_\epsilon * U _i )U_j + W_i U_j + U_i W_j + W_i W_j},
\end{equation}
where $R_i$ denote the Riesz transforms.  It also satisfies \eqref{p.eq}.
We claim that $p_\e=\tilde p_\e$ up to an additive constant by proving that $\nabla(p_\e - \tilde p_\e)=0$.  To this end we use a well known fact about the forced, non-stationary Stokes system on  $\R^3\times [t_1,t_2]$ where $t_1<t_2$ are given points in time: if $g\in L^\infty(t_1,t_2;H^{-1}(\R^3))$ and $V_0\in L^2(\R^3)$, then there exists a unique $ V\in  C_w([t_1,t_2];L^2(\R^3))\cap L^2(t_1,t_2;H^1(\R^3))$ and unique $\nabla  \pi $ satisfying $V(x,t_1)=V_0(x)$ and
\[(\partial_t  V -\Delta  V +\nabla \pi)(x,t) = g(x,t),\qquad\div V(x,t)=0,\] for $(x,t)\in \R^3\times [t_1,t_2]$. Formulas for $ V$ and $ \pi $ can be written using the Green tensor for time dependent Stokes system.  We only need the  uniqueness and recall its proof: Assume $(\hat V,\hat \pi)$ is a second solution. Then $V-\hat V$ and $\pi-\hat \pi$ satisfy the unforced Stokes system and, testing against $V-\hat V$, we obtain 
\begin{align*}
\frac 1 2 \int |(V-\hat V)(x,t)|^2\,dx + \int_0^t\int |\nabla (V-\hat V)(x,t')|^2\,dx\,dt' \leq 0,
\end{align*}implying $V=\hat V$.  This implies also that $\nabla \pi =\nabla \hat \pi$.

For our purposes let $V(x,t)=(2t)^{-1/2}U_\epsilon(y,s)$,
$ \pi(x,t) = (2t)^{-1} p_\epsilon(y,s)$, and $g(x,t)=(g_1+g_2)(x,t)$ where
\begin{align}
& g_1(x,t)= -\frac 1 {\sqrt{2t}^3}\mathcal (LW)(y,s),
\\&  g_2(x,t)= -\frac 1 {\sqrt{2t}^3}(W\cdot \nabla U_\epsilon + U_\epsilon \cdot\nabla W +(\eta_\epsilon*U_\epsilon) \cdot \nabla U_\epsilon +W\cdot\nabla W\big)(y,s),
\end{align}
and $y=x/\sqrt{2t}$ and $s=\log(\sqrt {2t})$.
Then, 
$g \in L^\infty(1,\lambda^2;H^{-1}(\R^3))$ and $(V,\pi)$ solves the Stokes system on $\R^3\times [1,\lambda^2]$ and $V$ is in the energy class.
 We conclude that $\nabla \pi $ is unique, and is given by Riesz transforms,
 \[
 \nabla \pi = \nabla (\Delta)^{-1} \div g_2,
 \]
 noting that $g_1$ is divergence free. %
 Since taking Riesz transforms commutes with the above change of variables and letting $\tilde \pi=(2t)^{-1}\tilde p_\epsilon$, we conclude that $\nabla \pi=\nabla \tilde \pi$, and hence $\nabla (p_\e-\tilde p_\e)=0$. We may therefore replace $p_\e$ by $\tilde p_\e$,  and apply the Calderon-Zygmund theory to obtain an \emph{a priori} bound for $p_\epsilon$, namely
\begin{equation}
\|p_\epsilon \|_{L^{5/3}(\mathbb R^3\times [0,T])}\leq C\|U_\epsilon \|_{L^{10/3}(\mathbb R^3\times [0,T])} ^2 +C\| {W} \|_{L^{10/3}(\mathbb R^3\times [0,T])}  ^2,
\end{equation}
which is finite and independent of $\epsilon$ by the known properties of $U_\epsilon$ and $W$ (using $q=10/3$).  

Because $U_\epsilon$ are bounded independently of $\epsilon$ in $L^\infty (0,T;L^2(\R^3))\cap L^2(0,T;H_0^1(\R^3))$, and $U_\e$ is a weak solution of \eqref{eq:mollifiedLeray} with $W$ bounded by Lemma \ref{lemma:W}, there exists a vector field $U\in L^\infty (0,T;L^2(\R^3))\cap L^2(0,T;H_0^1(\R^3))$ and a sequence $\{ U_{\epsilon_k}\}$ of elements of $\{U_\epsilon \}$ so that
\begin{align*}
& U_{\epsilon_k} \rightarrow U \mbox{~weakly in}~L^2(0,T;X)
\\& U_{\epsilon_k}\rightarrow U \mbox{~strongly in}~L^2(0,T;H(K)) ~ \forall \mbox{~compact sets $K\subset \R^3$}
\\& U_{\epsilon_k}(s)\rightarrow U(s) \mbox{~weakly in}~L^2 \mbox{~for all}~s\in [0,T],
\end{align*}
as $\epsilon_k\to 0$.  Let $u=U+W$.  Furthermore, since $p_{\epsilon_k}$ are uniformly bounded in $L^{5/3}(\R^3\times [0,T])$ we can extract a subsequence (still denoted $p_{\epsilon_k}$) so that 
\begin{equation}
p_{\epsilon_k}\rightarrow p \mbox{~weakly in}~L^{5/3}(\R^3\times [0,T]),
\end{equation}
for some distribution $p\in L^{5/3}(\R^3\times [0,T])$ and this convergence is strong enough to ensure that $(u,p)$ solves \eqref{eq:wholeSpaceLeray} in the distributional sense.

It remains to check that the pair $(u,p)$ is suitable. This follows as in \cite[Appendix]{CKN} since the approximating solutions $(u_\e,p_\e)$ \emph{all satisfy the local energy equality}.  
\end{proof}

\section{$\lambda$-DSS initial data and the heat equation}
\label{sec3}

In this section we provide estimates for solutions to the heat equation when the initial data $v_0$ is divergence free, $\lambda$-DSS, and belongs to $L^3_w(\R^3)=L^{(3,\infty)}(\R^3)$.  Throughout this section let $V_0(x,t)=e^{t\Delta}v_0(x)$ and $U_0(y,s)= {\sqrt {2t}} (e^{t\Delta}v_0)(x)$ where $x,t,y,s$ satisfy \eqref{variables}.   

Generally, functions in $L^3_w(\R^3)$ can possess arbitrarily many singularities of order $|x|^{-1}$. This is false if the function is discretely self-similar.  In this case, the only critical singularity is at the origin; any other singularities must be subcritical. This is clarified in the following lemma.
\begin{lemma}\label{lemma:equivalence}
If  $f$ is defined in $\R^3$ and is $\lambda$-DSS for some $\lambda>1$, then  $f \in L^3_{loc}(\R^3 \setminus \{0\})$ if and only if $f \in L^{3}_w(\R^3)$. 
\end{lemma}

\begin{proof}
Let 
\begin{equation}
A_r = \bket{x \in \R^3: r \le |x| < r \lambda }.
\end{equation}
Decompose $f = \sum_{k \in \mathbb{Z}} f_k$ where $f_k(x) = f(x)$ if $x \in A_{\lambda^k }$, and $f_k(x)=0$ otherwise. Note $f_k (x) = \lambda^{-k} f_0(\lambda^{-k}x)$ since $f$ is $\lambda$-DSS.

The distribution function for $f$ is
\[
m(\si,f) = |\{ x: |f(x)| > \si \}|.
\]
Recall the identity
\[
\int |f|^p \,dx = p \int_0^\I \si^{p} m(\si,f) d \si/\si,
\]
which holds for $1\le p < \I$. For $\beta>0$, 
\EQ{\label{eq3.2}
m(\beta, f) &= \sum_{k \in \mathbb{Z}} m(\beta,f_k)
= \sum_{k \in \mathbb{Z}} m(\lambda^k \beta,f_0) \lambda^{3k},
}
where we have used the scaling property $f_k (x) = \lambda^{-k} f_0(\lambda^{-k}x)$. 
However,
\EQ{
\int_{A_1} |f_0|^p\,dx
& =  p \int_0^\I \si^{p} m(\si,f_0) d \si/\si
\\
& = \sum_{k \in \Z} p \int_{\lambda^{k-1}\beta} ^{\lambda^k \beta} \sigma^{p-1} m(\si,f_0) d \si
\\
& \ge  \sum_{k \in \Z} p \int_{\lambda^{k-1}\beta} ^{\lambda^k \beta} (\lambda^{k-1}\beta) ^{p-1} m(\lambda^{k}\beta,f_0) d \si
\\
& = \sum_{k \in \Z} p (\lambda^k \beta- \lambda^{k-1}\beta) (\lambda^{k-1}\beta) ^{p-1} m(\lambda^{k}\beta,f_0) 
\\
& = \beta^p p (\lambda-1)\lambda^{-p} \sum_{k \in \Z} \lambda^{kp} m(\lambda^{k}\beta,f_0) .
}
Thus, with the choice $p=3$ and using \eqref{eq3.2}, we get
\EQ{
m(\beta, f) &\leq \frac {\lambda^3}{\beta^3 3 (\lambda-1)} \int_{A_1} |f_0|^3\,dx.
}
Since $\beta>0$ is arbitrary, we conclude
\EQ{
\norm{f}_{L^{3}_w(\R^3)} ^3 \leq  \frac {\lambda^3}{ 3 (\lambda-1)} \int_{A_1} | f|^3\,dx.
} 

On the other hand,
\EQ{
\int_{A_1} |f_0|^p\,dx
& =  p \int_0^\I \si^{p} m(\si,f_0) d \si/\si
\\
& = \sum_{k \in \Z} p \int_{\lambda^{k}\beta}^{\lambda^{k+1} \beta} \si^{p-1} m(\si,f_0) d \si
\\
& \leq  \sum_{k \in \Z} p \int_{\lambda^{k}\beta} ^{\lambda^{k+1} \beta} (\lambda^{k+1}\beta) ^{p-1} m(\lambda^{k}\beta,f_0) d \si
\\
& = \sum_{k \in \Z} p (\lambda^{k+1} \beta- \lambda^{k}\beta) (\lambda^{k+1}\beta) ^{p-1} m(\lambda^{k}\beta,f_0) 
\\
& = \beta^p p (\lambda-1)\lambda^{p-1} \sum_{k \in \mathbb{Z}} \lambda^{kp} m(\lambda^{k}\beta,f_0) .
}
Thus, with $p=3$ and using \eqref{eq3.2}, we get
\EQ{
 \int_{A_1} |f_0|^ 3 \,dx\le  3(\lambda-1)\lambda^{2}  \beta^3  m(\beta, f) \le 3(\lambda-1)\lambda^{2} \norm{f}_{L^{3}_w(\R^3)} ^3 .
}
\end{proof}
 
Our next lemma concerns the decay at spatial infinity for times bounded away from $t=0$ of solutions to the heat equation with discretely self-similar $L^3_w$ data.
\begin{lemma}\label{lemma:V0decay}
Suppose  $v_0 \in L^3_w(\R^3 \backslash \{0\})$ is $\la$-DSS for some $\la>1$ and let $V_0$ be defined as above. Then,  
 \[
 \label{w1.decay}
 \sup_{1\leq t\leq \lambda^2}\norm{V_0(t)}_{L^{q}(|x|>R)}\leq \Theta(R),
 \]
for any $q \in (3,\I]$ and $t\in [1,\lambda^2]$ where
$\Theta:[0,\infty)\to [0,\infty)$ depends on $q$ but satisfies $\Theta(R)\to 0$ as $R\to\infty$.
\end{lemma}

\begin{proof} By Lemma \ref{lemma:equivalence} we have $v_0\in L^3_{loc}(\R^3\setminus \{0\})$.
Let
\[
\om(r) = \sup_{1<|x_0|<\la} \int_{B(x_0,r)} |v_0|^3\,dx.
\]
Clearly $\om(r) \to 0$ as $r \to 0$.   Let $A_R=\{ x: R\leq |x| <\lambda R \}$.  
We first establish a general estimate for $\|V_0(t)\|_{L^q(A_R)}$ which we will then sum over nested shells.  Let
\EQ{
V_0 (x,t)
&= \int_{|z|< R/2}  {(4\pi t)^{-3/2}}e^{-|x-z|^2/{2t}}v_0(z) \,dz 
\\&\quad +\int_{R/2\leq |z|< 2\lambda R} (4\pi t)^{-3/2}e^{-|x-z|^2/{2t}}v_0(z)\,dz 
\\&\quad + \int_{2\lambda R\leq |z| } (4\pi t)^{-3/2}e^{-|x-z|^2/{2t}}v_0(z)  \,dz
\\&= I_0^R(x,t)+I_1^R(x,t)+I_2^R(x,t).
} 
Fix $(x,t)\in A_R\times [1,\lambda^2]$. Then, 
\[ |I_0^R(x,t)|+|I_2^R(x,t)|\lesssim e^{-cR^2}.\]
Hence $I_0^R,\,I_2^R\in L^p (A_R)$ with 
\[
	\norm{I_0^R}_{L^p(A_R)}+\norm{I_2^R}_{L^p(A_R)}\leq ce^{-cR^2} R^{3/p},
\]
for all $1\leq p \leq \infty$.  

We further decompose $I_1^R$ as
\begin{align*}
I_1^R(x,t) &=\bket{ \int_{z \in A_R^* , |z-x| < \th R} + \int_{z \in A_R^*, |z-x| > \th R} }(4\pi t)^{-3/2}e^{-\frac {|x-z|^2}{2t}} v_0(z)dz
\\& = : I_3^R(x,t) + I_4^R(x,t),
\end{align*}
where $0< \th \ll 1$ is an as-of-yet unspecified parameter and $A_R^*=\{ z: R/2\leq |z| <2\lambda R \}$.
We have by H\"older's inequality that
\[
|I_3^R(x,t)| \le C \norm{e^{-cx^2}}_{L^{3/2}(\R^3)} \norm{ v_0}_{L^3(B(x,\th R))} \le C \om(\th),
\]
and 
\EQ{
|I_4^R(x,t)| & \le C \int _{A_R^*} e^{-c \th^2 R^2}  | v_0(z)| \,dz 
\\
&\le C e^{-c \th^2 R^2}  \norm{ v_0}_{L^3(A_R^*)}  \norm{ 1}_{L^{3/2}(A_R^*)} 
\le C e^{-c \th^2 R^2}  R^2.
}
Therefore, for $R>1$, (and we drop the $t$ dependence of $V_0$ below)
\EQ{
\norm{V_0}_{L^\I(A_R)} \le  C \om(\th) + C e^{-c \th^2 R^2}  R^2+ Ce^{-c R^2},
}
where the constants are independent of $R$ and $\theta$.  The above inequality is still valid if  $\lambda^kR$ replaces $R$ for $k\in \N$, indeed we have 
\EQ{
\norm{V_0}_{L^\I(A_{\lambda^kR})} \le  C \om(\th) + C e^{-c \th^2 (\lambda^kR)^2}  (\lambda^kR)^2+ Ce^{-c (\lambda^kR)^2}.
}
The right hand side is decreasing in $k$ for fixed $\theta$ and $R$ and we conclude that
\[
 \norm{V_0}_{L^\infty({|x|\geq R})}\leq 
\sup_{k\in \N} \norm{V_0}_{L^\I(A_{\lambda^kR})} \leq C \om(\th) + C e^{-c \th^2 R^2}  R^2+ Ce^{-c R^2}.
\]
If $q\in (3,\I)$ we have
\begin{align*}
\norm{V_0}_{L^q(|x|\geq R)} &\leq C \norm{V_0}_{L^\I(|x|\geq R)}^{1-3/q} \norm{V_0}_{L^{3}_w}^{3/q} \\&\leq C ( C\om(\th) +   Ce^{-c \th^2 { R}^2}  { R}^2+ C e^{-c { R}^2})^{1-3/q} \norm{V_0}_{L^{3}_w}^{3/q}
.\end{align*}

We now construct $\Theta(R)$. Let $\epsilon_k=2^{-k}$ for $k\in \N$. For each $\epsilon_k$, choose $\th_k>0$ sufficiently small so that 
\[C \om(\th_k) \leq \frac  {\e_k^{q/(q-3)}} {2 C^{q/(q-3)}\| V_0 \|_{L^{3}_w}^{3/(q-3)}}.\] Then choose $R_k$ sufficiently large so that $R_k>R_{k-1}$ and
\[ C e^{-c \th_k^2 R_k^2}  R_k^2 + Ce^{-c R_k^2} \leq \frac  {\e_k^{q/(q-3)}} {2C^{q/(q-3)}\| V_0 \|_{L^{3}_w}^{3/(q-3)}}.\]
Finally, let
\[
\Theta (R)=
\begin{cases}
1 &\text{if } 0<R<R_1
\\ \epsilon_k &\text{if } R_k\leq R< R_{k+1}
\end{cases},
\]
which completes our proof.
\end{proof}
\begin{remark}
(i)
The decay rate in Lemma \ref{lemma:V0decay} depends not only on $\norm{v_0}_{L^3(A_1)}$, but also on $\om(r)$, see \eqref{w1.decay} above. It is worth noting that there is no decay rate that applies to all $v_0$ bounded in $L^{3}_w$.  Indeed, there is a constant $c_0>0$, a point $x_ 0 \in A_1$, and a sequence of $\la$-DSS  $v_0^k\in L^{3}_w$, $k \in \N$, such that $\norm{v_0^k}_{L^{3}_w(\R^3)} \le 1$ and, for all $k$ sufficiently large,
\[
 \inf_{ B(x_k,1)} |V_0^k| \ge c_0, \quad x_k = \la^k x_0 .
\] In particular, choose any $x_0 \in A_1$ not on its boundary, and $r_0>0$ so that $B(x_0,r_0) \subset A_1$. For any integer $k\ge \log_{\la} r_0^{-1}$, we have $\la^{-k} \le r_0$. Let $v_0^k(x)=0$ for $ x \in A_1 \backslash B(x_0,\la^{-k})$ and $v^k_0(x) = c_1\la^k$ if $x \in B(x_0, \la^{-k})$ for some constant $c_1>0$.
Then $\norm{v_0^k}_{L^{3,\I}(\R^3)} \le C \norm{v_0^k}_{L^3(A_1)}\le 1$ for suitable choice of $c_1$ independent of $k$.
We have  $v_0 = c_1$ in $B(x_k,1)$,
$x_k = \la^k x_0$. Thus, for $ x \in B(x_k,1)$, 
\[
V_0(x,t) \ge \int_{B(x_k,1)} e^{-4c} v_0(y)\,dy=c_0: = \frac {4\pi}3  e^{-4c}c_1.
\]

(ii) If we assume more regularity on $v_0$, then we can get explicit decay rate of $V_0$. For example, if $v_0 \in L^q(A_1)$, 
 $ 3<q \le \I$, then for all $ 3<p \le \I$,
\begin{equation}\label{eq3.12}
\sup_{t \in [1,\lambda^2]} \norm{ V_0(\cdot,t)}_{L^p(|x|>R)}  \le C \norm{v_0}_{L^q(A_1)} R^{-\sigma}
\quad
\forall R\gg 1,
\end{equation}
where $\sigma = 1-3/q$ for $p \in [q,\I]$, and $\sigma = 1-3/p$ for $3<p<q$. This shows that our assumption $v_0 \in L^3(A_1)$ is in borderline. Note \eqref{eq3.12} does not depend on $\omega$-like functions.
The proof of \eqref{eq3.12} is omitted since it is not used.
\end{remark}

The main lemma of this section connects solutions of the heat equation to the boundary data characterized by  Assumption \ref{AU_0}.
\begin{lemma}\label{th:2.1}
Suppose $v_0$ satisfies the assumptions of Theorem \ref{thrm:main} and let $x,t,y,s$ satisfy \eqref{variables}.  Then
\begin{equation}\label{def:U0} 
U_0(y,s)= {\sqrt {2t}} (e^{t\Delta}v_0)(x), 
\end{equation}satisfies Assumption \ref{AU_0} with $T=\log \lambda$ and any $q \in (3,\I]$.
\end{lemma}
\begin{proof}
Since $v_0$ is divergence free and $\lambda$-DSS, $e^{t\Delta}v_0$ is the divergence free, $\lambda$-DSS solution to the heat equation for $(x,t)\in \R^3\times [0,\infty)$.  Under the change of variables \eqref{variables}, it follows that $U_0$ is divergence free, $T$-periodic for $T=\log \lambda$, and satisfies
\begin{equation}\label{eq:periodicU0}
LU_0= \partial_s U_0(y,s)-\Delta U_0(y,s) -U_0(y,s)-y\cdot\nabla U_0(y,s) =0,
\end{equation}
for all $(y,s)\in \R^3\times \R$. Inclusion in $C^1$ comes from the smoothing effect of the heat kernel. 
This also implies that $\partial_s U_0\in L^\infty(0,T;L_{loc}^{6/5}(\R^3))$.  
By Lemma \ref{lemma:V0decay} we know $U_0\in L^\infty(0,T;L^q(|x|>1) )$ and, since it is in $C^1(\R^4)$, it is also in $L^\infty(0,T;L^q(|x|\leq 1) )$. Hence $U_0\in L^\infty(0,T;L^q(\R^3) )$.
 The last bound in Assumption \ref{AU_0} is a direct consequence of Lemma \ref{lemma:V0decay}.
\end{proof}

\section{Discretely self-similar solutions to 3D NSE}\label{sec:DSS}

In this section we prove Theorem \ref{thrm:main}. 

\begin{proof}[Proof of Theorem \ref{thrm:main}]
By Lemma \ref{th:2.1}, $U_0(y,s)$ defined by \eqref{def:U0} satisfies Assumption \ref{AU_0}.
Let $(u,p)$ be the time-periodic weak solution described in Theorem \ref{thrm:existenceOnR3}.
Let $v(x,t)= u(y,s)/\sqrt{2t}$ and $\pi(x,t)=p(y,s)/2t$ where $y= x/\sqrt{2t}$ and $s=\log (\sqrt{2t})$.
Then $(v,\pi)$ is a distributional solution to \eqref{eq:NSE}. Indeed, if we let $\zeta(x,t) = \frac 1{2t} f(y,s)$ where $f(y,s)$ is the test vector 
in the weak form \eqref{u.eq-weak} of the $u$-equation, and note that
\[
\pd_t \zeta(x,t) = \frac 1{(2t)^2} (\pd_s -2 - y \cdot \nb _y)f(y,s),
\]
we recover the weak form of the Navier-Stokes equations \eqref{eq:NSE} for $v$ with test vector $\zeta$ from
the weak form  \eqref{u.eq-weak} for $u$.

Note
\begin{equation} \notag  v-e^{t\Delta}v_0 \in L^\infty(1,\lambda^2;L^2(\R^3))\cap L^2(1,\lambda^2;H^1(\R^3)). \end{equation}
The $\lambda$-DSS scaling property implies
\begin{equation}\notag 
||v(t)-e^{t\Delta}v_0||_{L^2(\R^3)}^2\lesssim t^{1/2} \sup_{1\leq \tau\leq \lambda^2} ||v(\tau)-e^{\tau \Delta}v_0||_{L^2(\R^3)}^2,
\end{equation}
and
\begin{equation}\notag
\int_0^{\lambda^2} \int || \nabla (v(t)-e^{t\Delta}v_0)||_2^2\,dx\,dt \lesssim \bigg( \sum_{k=0}^\infty  \lambda^{-k}  \bigg) \int_1^{\lambda^2}\int || \nabla (v(t)-e^{t\Delta}v_0)||_2^2\,dx\,dt.
\end{equation}
It follows that
\begin{equation}\label{ineq:time0}  v-e^{t\Delta}v_0 \in L^\infty(0,\lambda^2;L^2(\R^3))\cap L^2(0,\lambda^2;H^1(\R^3)). \end{equation}
We now check that $v$ is a local Leray solution to \eqref{eq:NSE}. 

\emph{Locally finite energy and enstrophy:} This follows from inequality \eqref{ineq:time0} noting that $v_0\in L^2_{u\,loc}$ implies $e^{t\Delta}v_0$ has uniformly locally finite energy and enstrophy.

\emph{Convergence to initial data:} The fact that $||v(t)-e^{t\Delta}v_0||_{L^2(\R^3)}\lesssim t^{1/4}$ implies convergence to zero in the  $L^2_{loc}(\R^3)$ norm.  Using the embedding $L^3_w\subset M^{2,1}\subset L^2_{-3/2}$ where $L^2_{-3/2}$ is the weighted $L^2$ space 
(see Comment 4 after Theorem \ref{thrm:main}  for the definition)
as well as the fact that $e^{t\Delta}v_0\to v_0$ in $L^{2}_{-3/2}(\R^3)$ (see \cite[Remark 3.2]{Kato}), and this space embeds in $L^2_{loc}$, we conclude that $e^{t\Delta}v_0\to v_0$ in $L^{2}_{loc}(\R^3)$ as $t\to 0^+$.  It follows that
\begin{equation}\notag
\lim_{t\to 0} ||v(t)-v_0||_{L^2_{loc}(\R^3)}=0. 
\end{equation}

\emph{ {Decay at spatial infinity:}} For any $R>0$, the $\lambda$-DSS scaling implies $v(t)-e^{t\Delta}v_0 \in L^2(0,R^2;\R^3)$.  Together with the fact that $e^{t\Delta}v_0(x)$ satisfies the same decay requirements at spatial infinity as a local Leray solution (this is easy to see given that $v_0\in L^2_{u\,loc}(\R^3)$), this implies that 
\begin{equation}\notag 
\lim_{|x_0|\to \infty} \int_0^{R^2}\int_{B_R(x_0 )} | v(x,t)|^2\,dx \,dt=0.
\end{equation}

\emph{Local energy inequality:} This property for $(v,\pi)$ is inherited from the suitability of $(u,p)$ in the self-similar variables.
Indeed,  if we let $\phi(x,t) = \frac 1{\sqrt{2t}} \psi(y,s)$ where $\psi(y,s)$ is the test function
in the local energy inequality \eqref{ineq:localEnergy} for $(u,p)$, 
 and note that
\[
\pd_t \phi(x,t) = \frac 1{(2t)^{3/2}} (\pd_s -1 - y \cdot \nb _y)\psi(y,s),
\]
we recover the local energy inequality \eqref{eq:localEnergyNSE} for $(v,\pi)$ with test function $\phi$.
\end{proof}

\begin{remark} In the definitions $\zeta(x,t) = \frac 1{2t} f(y,s)$ and  $\phi(x,t) = \frac 1{\sqrt{2t}} \psi(y,s)$ in the above proof,
the exponents of $\sqrt {2t}$ in front of  $f$ and $\psi$ can be understood by dimension analysis.
The dimensions of the physical variables $x,t,v,\pi$ are $1,2,-1,-2$ respectively and are reflected in the ansatz \eqref{ansatz2}.  \emph{All self-similar variables $y,s,u,p,f,\psi$ are dimension free}, hence so are the weak form and local energy inequality for $(u,p)$ and therefore those for $(v,\pi)$. As a result, the dimension of $\zeta$ is $-2$, and that of $\phi$ is $-1$. Also note that
$\zeta$ should have the same dimension as $\phi v$, which is correct.
\end{remark}

\section{On existence of self-similar solutions}\label{sec:NSE}
\label{sec5} 

As mentioned in the comments prior to the statement of Theorem \ref{thrm:selfsimilardata}
our ideas can be extended to give  simple proofs of the existence of self-similar solutions to the 3D Navier-Stokes equations with $(-1)$-homogeneous initial data $v_0$ satisfying \eqref{ineq:decayingdata}. 
Recall that the first construction is given in  \cite{JiaSverak} and the second in \cite{KT-SSHS}. In the first subsection below we show that we can get self-similar solutions as limits of $\lambda_k$-DSS solutions with $\lambda_k \to 1_+$ as $k \to \I$. %
In the second subsection we present a third construction of self-similar solutions following and simplifying the ideas of the previous sections: It  constructs solutions to stationary Leray equations directly, based on our new explicit a priori bound.

\subsection{Self-similar solutions as limits of DSS solutions}
\begin{proof}[Proof of Theorem \ref{thrm:selfsimilardata}] Let $v_0$ be $(-1)$-homogeneous and satisfy the assumptions of Theorem \ref{thrm:main}.  Then, $v_0$ is $\lambda$-DSS for every factor $\lambda>1$. 
For $k\in \mathbb N$, let $\lambda_k=2^{(2^{-k})}$ so that $\lambda_{k+1}^2 = \lambda_k$.  
This sequence  decreases strictly to $1$ as $k\to\infty$. 
Let $v_k$ be the $\lambda_k$-DSS local Leray solution obtained from Theorem \ref{thrm:main} (or from \cite[Theorem 1.1]{Tsai-DSSI}) with scaling factor $\lambda_k$.  Working within the local Leray class provides \emph{a priori} bounds for all $v_k$. In particular,  letting $\mathcal N (v_0)$ denote the class of local Leray solutions with initial data $v_0$, the following estimate is well known for local Leray solutions (see \cite{JiaSverak}): for all $\tilde v\in \mathcal N (v_0)$ and $r>0$ we have
\begin{equation}
\esssup_{0\leq t \leq \sigma r^2}\sup_{x_0\in \RR^3} \int_{B_r(x_0)}\frac {|\tilde v|^2} 2 \,dx\,dt + \sup_{x_0\in \RR^3}	\int_0^{\sigma r^2}\int_{B_r(x_0)} |\nabla \tilde v|^2\,dx\,dt <C \sigma ,
\end{equation}
where \begin{equation} \sigma(r) =c_0\, \min\bigg\{r^2\bigg( \sup_{x_0\in \RR^3} \int_{B_r(x_0)} \frac {|v_0|^2}2\,dx \bigg)^{-2} , 1  \bigg\},
\end{equation}
for a small universal constant $c_0$.  Note that $v_0$ belongs to the Morrey space $M^{2,1}$ by the embedding $L^3_w\subset M^{2,1}$.  This implies that $\sigma(r)r^2 \to \infty$ as $r\to \infty$ and, since $v_k\in \mathcal N (v_0)$, we obtain \emph{a priori} bounds for all $v_k$ across the time interval $[0,2]$ which are independent of $k$.  This allows us to pass to the limit to obtain a local Leray solution $v\in \mathcal N (v_0)$.  Then, for any $k$, and sufficiently large $l$, it follows that $v_l$ is DSS with scaling factor $\lambda_k$, and this property is inherited by $v$. 

To show that $v$ is SS we pass to the time periodic variables $y=x/\sqrt{2t}\in\RR^3$ and $s=\log\sqrt{2t}\in \RR$ and write $u_k(y,s)=\sqrt{2t} v_k(x,t)$ where $u_k$ is time periodic with period $T_k=\log \lambda_k$, $T_{k+1}=\frac 12 T_k$. Note that all $u_l$, $l>k$, are $T_k$ periodic, and this property is inherited by $u$. 
Let $U_0$ be defined as in  \eqref{def:U0} which is now constant in $s$. Note that $u$ is $T_k$ periodic for all $k$, and $u(y,s)-U_0(y)$ is weakly continuous $L^2(\R^3)$-valued vector fields. Hence $u$ must be constant in $s$.

Therefore $u$ solves the stationary Leray equations, which proves that $v$ is a self-similar solution on $\R^3\times (0,\infty)$.
\end{proof}

\subsection{Third construction}
Alternatively, we may adapt the approach of Sections \ref{sec2}-\ref{sec:DSS} to the stationary Leray system and construct self-similar solutions directly without involving the DSS class.  
\begin{proof}[Proof of Theorem \ref{thrm:selfsimilardata}]
Let $U_0$ and $W$ be defined as in Sections  \ref{sec2} and  \ref{sec3}.  Then, $W$ satisfies the estimates \eqref
{ineq:Wsmall}--\eqref{LW.est} with $q=\I$, and
 is constant in the time variable.  Our first goal is to find a divergence free function $u\in L^2_{u\,loc}(\R^3)$ satisfying
\begin{align}\label{eq:stationary1}
&   -(\nabla u,\nabla f)+(u+y\cdot\nabla u-u\cdot\nabla u ,f)    =0,
\end{align}
for all $f\in \mathcal V$, and achieve this by solving a perturbed system for $U=u-W$.  The variational form of the perturbed, stationary Leray system is
\begin{align}\label{eq:stationary2}
		- (\nabla U,\nabla f) 
		+ (U+y\cdot \nabla U, f)
		- (U \cdot\nabla U, f)
=
		(W\cdot\nabla U+U\cdot \nabla W,f)+\langle \mathcal{R}(W),f\rangle,
\end{align}
which should hold for all $f\in \mathcal V$.  Solutions to this system can be approximated by a Galerkin scheme, the elements of which are obtained via Brouwer's fixed point theorem. In particular, let $\{ a_k \}\subset \mathcal V$ be an orthonormal basis of $H$.  For $k \in \mathbb N$, the approximating solution \[
U_k(y)=\sum_{i=1}^k b_{ki}a_i(y),
\] is required to satisfy
\begin{align}\label{eq:stationaryODE}
 & \sum_{i=1}^k A_{ij}b_{ki} +\sum_{i,l=1}^k B_{ilj} b_{ki}b_{kl} +C_j=0,%
\end{align}
for $j\in \{1,\ldots,k\}$,
where
\begin{align}
\notag A_{ij}&=- (\nabla a_{i},\nabla a_j) 
		+ (a_i+y\cdot \nabla a_i, a_j) 
		 -(    a_i\cdot \nabla W,a_j)
		- (W\cdot\nabla a_i, a_j)
\\\notag B_{ilj}&=- (a_i \cdot\nabla a_l, a_j)
\\\notag C_j&=-\langle \mathcal R (W),a_j\rangle.
\end{align}
Let $P(x):\R^k\to\R^k$ denote the mapping  
\[
P(x)_j=\sum_{i=1}^k A_{ij}x_{i} +\sum_{i,l=1}^k B_{ilj} x_{i}x_{l} +C_j.
\]
For $x\in \R^k$ let $\xi=\sum_{j=1}^k x_j a_j$.  We have
\EQ{
\label{eq4.6}
P(x)\cdot x &= -\frac 1 2 ||\xi||_{L^2}^2
	 -   \frac 1 2 ||\nabla \xi||_{L^2}^2 +(\xi \cdot \nabla \xi,W) - \bka{\cR(W),\xi}
\\
& \le  -\frac 1 4 ||\xi||_{L^2}^2  -   \frac 1 4 ||\nabla \xi||_{L^2}^2 +C_*^2 \norm{\cR(W)}_{H^{-1}}^2
\\
& \le  -\frac 1 4 |x|^2 + C_*^2 \norm{\cR(W)}_{H^{-1}}^2,
}
using the smallness of $\norm{W}_{L^\I}$. We conclude that
\[
P(x)\cdot x< 0,\quad \text{if } |x|=\rho := 3C_* \norm{\cR(W)}_{H^{-1}}.
\]
By Brouwer's fixed point theorem, there is one $x$ with $|x|<\rho$ such that $P(x)=0$. (Note this $\rho$ is independent of $k$. This is a feature of the Leray system, not of the Navier-Stokes.) Then $U_k=\xi$ is our approximation solution satisfying
\eqref{eq:stationaryODE}, with  \emph{a priori} bound
\[
 \norm{U_k}_{L^2}^2 + \norm{\nabla U_k}_{L^2} ^2\le 4C_*^2 \norm{\cR(W)}_{H^{-1}}^2,
\]
by the first inequality of \eqref{eq4.6} and $P(x)=0$.
This bound is sufficient to find a subsequence with a weak limit in $H^1(\R^3)$ and a strong limit in $L^2(K)$ for any compact set $K$ in $\R^3$, that is, there exists a solution $U$ to \eqref{eq:stationary2} which satisfies $U\in H^1(\R^3)$. A solution to \eqref{eq:stationary1} is now obtained by setting $u=U+W$.  Note that $u\in H^1_{loc}\cap L^q$, $3<q\le 6$, and, following \cite[pp. 287-288]{NRS} or \cite[pp. 33-34]{Tsai-ARMA}, if we define
\[
p = \sum_{i,j} R_i R_j (u_i u_j),
\]
where $R_i$ denotes the Riesz transforms,
then $(u,p)$ solve the stationary Leray system in the distributional sense and, furthermore, by Calderon-Zygmund estimates,
\[
||p||_{L^{q/2}(\R^3)}<C ||u||_{L^q(\R^3)}^2,\quad (3<q\le 6).
\]

A solution pair $(v,\pi)$ to \eqref{eq:NSE} is now obtained by passing from the self-similar to the physical variables at time $t=1/2$ and extending to all times using the self-similar scaling relationships.  It remains to show that $(v,\pi)$ is a local Leray solution. 

A regularity result for a generalized stationary Stokes system (see \cite[Proposition 1.2.2]{Temam}) leads to higher regularity for the pair $(U,p)$ on compact subsets of $\R^3$.  In particular, $U$ and $p$ are infinitely differentiable.  Since $W,\, U\in C^\infty$, so is $u$.  This guarantees that $v$ and $\pi$ are smooth in the spatial variables.  Smoothness in time is apparent from the self-similar scaling properties of $v$ and $\pi$. Therefore, testing against the equation for $v$ against $\phi\,v $ where $\phi\in C_0^\infty ( \R^3\times (0,\infty))$ is non-negative and integrating by parts confirms that the pair $(v,\pi)$ satisfies the local energy identity.  The remaining conditions from Definition \ref{def:localLeray} follow as in the proof of Theorem \ref{thrm:main}.
\end{proof}

\section*{Acknowledgments}
The research of
both authors was partially supported by the Natural Sciences and
Engineering Research Council of Canada grant 261356-13.

Zachary Bradshaw, Department of Mathematics, University of British
Columbia, Vancouver, BC V6T 1Z2, Canada;
e-mail: zbradshaw@math.ubc.ca

\medskip

Tai-Peng Tsai, Department of Mathematics, University of British
Columbia, Vancouver, BC V6T 1Z2, Canada;
e-mail: ttsai@math.ubc.ca


\begin{thebibliography}{XX}
\bibitem{BCI} Biryuk, A., Craig, W. and Ibrahim, S., Construction of suitable weak solutions of the Navier-Stokes equations. Stochastic analysis and partial differential equations, 1-18, Contemp. Math., 429, Amer. Math. Soc., Providence, RI, 2007. 
\bibitem{CKN} Caffarelli, L., Kohn, R. and Nirenberg, L., Partial regularity of suitable weak solutions of the Navier-Stokes equations. Comm. Pure Appl. Math. 35 (1982), no. 6, 771-831.
%
\bibitem{CP} Cannone, M. and Planchon, F., Self-similar solutions for Navier-Stokes equations in $\R^3$. Comm. Partial Differential Equations 21 (1996), no. 1-2, 179-193. 
\bibitem{Galdi} Galdi, G. P., \emph{An introduction to the mathematical theory of the Navier-Stokes equations. Vol. II. Nonlinear steady problems.} Springer Tracts in Natural Philosophy, 39. Springer-Verlag, New York, 1994.
\bibitem{GS06} Galdi, G. P. and Silvestre, A. L., Existence of time-periodic solutions to the Navier-Stokes equations around a moving body. Pacific J. Math. 223 (2006), no. 2, 251-267.
\bibitem{GiMi} Giga, Y. and Miyakawa, T., Navier-Stokes flows in $\R^3$ with measures as initial vorticity and the Morrey spaces, Comm. Partial Differential Equations 14 (1989), 577-618.
%
%
%
\bibitem{Grujic} Gruji\'c, Z., Regularity of forward-in-time self-similar solutions to the 3D NSE, Discrete Contin. Dyn. Syst. 14 (2006), 837-843.
\bibitem{hopf} Hopf, E., \"Uber die Anfangswertaufgabe f\"ur die hydrodynamischen Grundgleichungen. Math. Nachr. 4, (1951). 213-231.
\bibitem{JiaSverak} Jia, H. and \v Sver\'ak, V., Local-in-space estimates near initial time for weak solutions of the Navier-Stokes equations and forward self-similar solutions. Invent. Math. 196 (2014), no. 1, 233-265.
\bibitem{KMT2012} Kang, K., Muira, H. and Tsai, T.-P., Asymptotics of small exterior Navier-Stokes flows with non-decaying boundary data. Comm. Partial Differential Equations 37 (2012), no. 10, 1717-1753.
\bibitem{Kato} Kato, T., Strong solutions of the Navier-Stokes equation in Morrey spaces. Bol. Soc. Brasil. Mat. (N.S.) 22 (1992), no. 2, 127-155.
\bibitem{KiSe} Kikuchi, N. and Seregin, G., Weak solutions to the Cauchy problem for the Navier-Stokes equations satisfying the local energy inequality. Nonlinear equations and spectral theory, 141-164, Amer. Math. Soc. Transl. Ser. 2, 220, Amer. Math. Soc., Providence, RI, 2007
\bibitem{Koch-Tataru}
Koch, H. and Tataru, D., Well-posedness for the Navier-Stokes equations.
Adv. Math. 157 (1), 22--35 (2001)
\bibitem{KT-SSHS} Korobkov, M. and Tsai, T.-P., Forward self-similar solutions of the Navier-Stokes equations in the half space, submitted.
\bibitem{LR} Lemari\'e-Rieusset, P. G., \emph{Recent developments in the Navier-Stokes problem.} Chapman Hall/CRC Research Notes in Mathematics, 431. Chapman Hall/CRC, Boca Raton, FL, 2002.
\bibitem{LuoTsai} Luo, Y. and Tsai, T.-P., Regularity criteria in weak $L^3$ for 3D incompressible Navier-Stokes equations, Funkcialaj Ekvacioj, to appear. arXiv:1310.8307
\bibitem{leray} Leray, J., Sur le mouvement d'un liquide visqueux emplissant l'espace. (French) Acta Math. 63 (1934), no. 1, 193-248. 
\bibitem{Morimoto} Morimoto, H., On existence of periodic weak solutions of the Navier-Stokes equations in regions with periodically moving boundaries, J. Fac. Sci. Univ. Tokyo Sect. IA Math. 18 (1971/72), 499-524.
\bibitem{NRS} Ne\v cas, J., {R\accent23 u\v {z}i\v {c}ka}, M., and {\v Sver\'ak}, V., On Leray's self-similar solutions of the Navier-Stokes
  equations, Acta Math. 176 (1996), 283--294.
\bibitem{Temam} Temam, R., \emph{Navier-Stokes equations. Theory and numerical analysis.} Reprint of the 1984 edition. AMS Chelsea Publishing, Providence, RI, 2001.
\bibitem{Tsai-ARMA} 
 Tsai, T.-P., On Leray's self-similar solutions of the
  Navier-Stokes equations satisfying local energy estimates, Archive
  for Rational Mechanics and Analysis 143 (1998), 29--51.
\bibitem{Tsai-DSSI}Tsai, T.-P., Forward discretely self-similar solutions of the Navier-Stokes equations. Comm. Math. Phys. 328 (2014), no. 1, 29-44.
\end{thebibliography}
\end{document}